\documentclass[10pt]{amsart}

\calclayout

\usepackage{amsmath, amssymb}
\usepackage{amsfonts}
\usepackage{mathrsfs}
\usepackage[color, arrow,matrix,curve,cmtip,ps]{xy}
\usepackage{paralist}
\usepackage{quiver}
\usepackage{amsthm}
\usepackage{caption}
\usepackage{tikz-cd}
\usetikzlibrary{arrows,calc,matrix}
\tikzset{
curvarr/.style={
  to path={ -- ([xshift=2ex]\tikztostart.east)
    |- (#1) [near end]\tikztonodes
    -| ([xshift=-2ex]\tikztotarget.west)
    -- (\tikztotarget)}
  }
}
\tikzset{%
    symbol/.style={%
        draw=none,
        every to/.append style={%
            edge node={node [sloped, allow upside down, auto=false]{$#1$}}}
    }
}

\usepackage[final]{pdfpages}
\usepackage{dsfont,txfonts}
\usepackage{tikz}
\usepackage{rotating}
\PassOptionsToPackage{hyphens}{url}
\usepackage{hyperref}
\usepackage{enumitem}
\usepackage{graphicx}
\usepackage{mathtools}
\usetikzlibrary{arrows,chains,matrix,positioning,scopes}

\usepackage[T1]{fontenc}
\usepackage[variablett]{lmodern}
\usepackage{xcolor}
\usepackage[citestyle=alphabetic,bibstyle=alphabetic]{biblatex}
\usepackage{lipsum}
\usepackage[most]{tcolorbox}

\newtheorem{theorem}{Theorem}
\theoremstyle{definition}

\newtheorem{proposition}[theorem]{Proposition}
\newtheorem{corollary}[theorem]{Corollary}

\newtheorem{remark}[theorem]{Remark}

\newtheorem{question}[theorem]{Question}

\newtheorem{example}[theorem]{Example}

\everymath{\displaystyle}
\allowdisplaybreaks

\newcommand{\Fun}{\text{Fun}}

\newcommand{\Sh}{\mathrm{Sh}}
\newcommand{\Cosh}{\mathrm{Cosh}}
\newcommand{\Sp}{\mathrm{Sp}}
\newcommand{\An}{\mathrm{An}}

\newcommand{\Cat}{\mathrm{Cat}}
\newcommand{\PR}{\mathrm{Pr}}

\newcommand{\Ind}{\mathrm{Ind}}
\newcommand{\perf}{\mathrm{perf}}

\newcommand{\cont}{\mathrm{cont}}

\newcommand{\op}{\mathrm{op}}
\newcommand{\dual}{\mathrm{dual}}

\newcommand{\Br}{\mathrm{Br}}

\DeclareFieldFormat{postnote}{#1}
\DeclareFieldFormat{multipostnote}{#1}

\begin{document}
\title{A remark on the invariance of $K$-theory under duality}
\author{Georg Lehner}
\begin{abstract}
In this short remark, we explain that two examples of invariance under duality for a localizing invariant $F$ hold purely formally when $F$ is $K$-theory, whereas the general statement for arbitrary localizing invariants does not reduce to a formal statement. We record a counterexample to the claim that the universal localizing invariant is invariant under the operation of taking opposite categories, originally due to Tabuada.
\end{abstract}
\maketitle

We will follow the general conventions on $K$-theory of dualizable categories as laid out in \cite{efimov2025ktheorylocalizinginvariantslarge} and \cite{krause_nikolaus_puetzstueck}. Let $\mathcal{C}$ be a dualizable $\infty$-category, and $F : \Cat^\perf \rightarrow \mathcal{E}$ a finitary localizing invariant in the sense of \cite{blumberg_gepner_tabuada}. In general, one does not have $\mathcal{C} \simeq \mathcal{C}^\vee = \mathrm{Fun}^L( \mathcal{C}, \Sp )$. However, in the following two examples, it still holds that \[F^{\cont}( \mathcal{C}) \simeq F^{\cont}( \mathcal{C}^\vee).\]

\begin{example} \label{example1}
It was remarked in the previous article \cite{lehner2026algebraicktheorystablycompact} by the author that for a stably locally compact space $X$ there is an equivalence
\[ F^{\cont}( \Sh( X, \Sp ) ) \simeq F^{\cont}( \Cosh( X, \Sp ) ),\]
where $ \Cosh( X, \Sp ) = \Fun^L( \Sh( X, \Sp ) , \Sp ) = \Sh( X, \Sp )^\vee$ is the $\infty$-category of cosheaves on $X$. This equivalence, in the case where $X$ is locally compact Hausdorff, is a trivial consequence of Verdier duality, but Verdier duality in the sense that $\Sh( X, \mathcal{C} ) \simeq \Cosh( X, \mathcal{C} )$ fails for stably locally compact spaces, and a version using the de Groot dual $X^\vee$ of $X$ needs to be used.
\end{example}

\begin{example} \label{example2}
If $P$ is a continuous poset, two independent calculations due to Efimov \cite[Section 5]{efimov2025ktheorylocalizinginvariantslarge} show that
\[ F^{\cont}( \Sh( P, \Sp ) ) \simeq \bigoplus_{p \in P^\omega} F(  \mathbb{S} ) \simeq F^{\cont}( \Cosh( P, \Sp) ).\]
We note that $\Cosh( P, \mathcal{C} ) \simeq \Fun^L( \Sh( P, \Sp), \mathcal{C} )$, see \cite[Section 21.1]{lurieSAG}.
\end{example}

Now consider the involution $(-)^{\op} : \Cat^\perf \simeq \Cat^\perf$, which sends a stable idempotent complete $\infty$-category to its opposite.

\begin{theorem}
There exists a natural equivalence of functors $ \Cat^\perf \rightarrow \Sp$,
\[ K \simeq K \circ (-)^{\op}. \]
\end{theorem}

\begin{proof}
Observe that the functor $(-)^{\simeq} : \Cat^\perf \rightarrow \An$ of taking the groupoid core is invariant under $(-)^{\op}$. Since $K$-theory is the universal localizing invariant equipped with a natural transformation $(-)^{\simeq} \rightarrow \Omega^\infty K$, see \cite{blumberg_gepner_tabuada}, the same holds for $K$-theory.
\end{proof}

This result extends immediately to dualizable stable $\infty$-categories. The functor $(-)^\vee : \PR^L_{\dual} \rightarrow \PR^L_{\dual}, \mathcal{C} \mapsto \mathcal{C}^\vee = \Fun^L(\mathcal{C}, \Sp)$ gives an involution, \cite[Section 1.7]{efimov2025ktheorylocalizinginvariantslarge}. If $\mathcal{C} = \Ind(\mathcal{C}_0)$ is compactly generated, then $\mathcal{C}^\vee = \Ind(\mathcal{C}_0^{\op})$, hence $(-)^\vee$ extends the involution $(-)^{\op}$ to dualizable $\infty$-categories.

\begin{corollary}
There exists a natural equivalence of functors $ \PR^L_{\dual} \rightarrow \Sp$,
\[ K^{cont} \simeq K^{cont} \circ (-)^\vee. \]
\end{corollary}

\begin{proof}
Every localizing invariant on  $ \PR^L_{\dual} $ is uniquely determined by its restriction via $\Ind :  \Cat^\perf \rightarrow \PR^L_{\dual} $, see \cite[Theorem 4.10]{efimov2025ktheorylocalizinginvariantslarge}. Therefore, the claim reduces to the previous theorem.
\end{proof}

\begin{example}
Let $R$ be an $E_1$-ring spectrum, or in particular an ordinary ring (unital, but not necessarily non-commutative). In this case the $\infty$-category $\mathrm{LMod}_R$ of left modules is compactly generated with $\mathrm{LMod}_R \simeq \mathrm{Ind}(\mathrm{Perf}_R)$, \cite[Proposition 7.2.4.2]{lurieha}. Write $K(R) = K(\mathrm{Perf}_R)$. Then
\[\mathrm{LMod}_R^\vee \simeq \mathrm{RMod}_R = \mathrm{LMod}_{R^{\op}},\] where $R^{\op}$ is the $E_1$-ring obtained from $R$ by reversing the order of multiplication, \cite[Proposition 7.2.4.3]{lurieha}. Hence it holds that $K(R) \simeq K(R^{\op})$.
\end{example}

However, the same claim fails for an arbitrary localizing invariant.

\begin{theorem} \label{noninvariance}
There exists a stable idempotent complete $\infty$-category $\mathcal{C}$ such that
\[ \mathcal{U}( \mathcal{C} ) \not\simeq \mathcal{U}( \mathcal{C}^{\op} )\]
where $\mathcal{U} : \Cat^\perf \rightarrow \mathrm{Mot}$ is the universal finitary localizing invariant.
\end{theorem}

The proof of Theorem \ref{noninvariance} will be given in the following section, and follows as a special case of the work of Tabuada, \cite{Tabuada2014Additive}. Before we give the proof, we would like to leave the reader with an open question:

\begin{question}
What is a practical criterion for a dualizable $\infty$-category $\mathcal{C}$ which holds for Examples \ref{example1} and \ref{example2} and which implies that $\mathcal{U}^{\cont}( \mathcal{C} ) \simeq \mathcal{U}^{\cont}( \mathcal{C}^\vee )$?
\end{question}

\noindent \textbf{Acknowledgments.} 
We would like to thank Christoph Winges for a discussion that led to this remark, and Vladimir Sosnilo for communicating the counterexample for Theorem \ref{noninvariance} to us. We also want to thank Maxime Ramzi and Thomas Nikolaus for helpful comments.

The author was funded by the Deutsche Forschungsgemeinschaft (DFG, German Research Foundation) – Project-ID 427320536–SFB 1442, as well as under Germany's Excellence Strategy EXC 2044/2 - 390685587, Mathematics Münster: Dynamics-Geometry-Structure.

\section*{A counterexample}

Let us recall some general facts about the Brauer group of a field $k$. A general reference is \cite[Section 2.4]{GilleSzamuely2006}. The Brauer group $\mathrm{Br}(k)$ of $k$ is an abelian group given by the set of equivalence classes of finite-dimensional central division algebras $D$ up to Brauer equivalence. The multiplication is defined via $[D][D'] = [D'']$ where $D''$ is the unique central division algebra such that
$ D \otimes_k D' \cong M_{n}(D''),$ where $M_{n}(D'')$ denotes the $n \times n$-matrix algebra of $D''$. The inverse of a central division algebra $D$ is given by $D^{\op}$.


The following theorem is crucial for the computation of Brauer groups of global fields.

\begin{theorem}[Albert--Brauer--Hasse--Noether, {\cite[Thm.~8.1.17]{NSW2008}}]
Let $K$ be a global field. Then there is an exact sequence
\[
0 \longrightarrow \Br(K)
  \longrightarrow \bigoplus_{v} \Br(K_v)
  \xrightarrow{\ \sum_v \operatorname{inv}_v\ }
  \mathbb{Q}/\mathbb{Z}
  \longrightarrow 0,
\]
where $v$ runs over all places of $K$, the first map is the sum of the
restriction maps $\Br(K)\to \Br(K_v)$, and $\operatorname{inv}_v$ is the local
Hasse invariant map. For a finite place $v$, the map
\[
\operatorname{inv}_v \colon \Br(K_v)\xrightarrow{\sim} \mathbb{Q}/\mathbb{Z}
\]
is an isomorphism; for a real place $v$, its image is
$\frac{1}{2}\mathbb{Z}/\mathbb{Z}\subset \mathbb{Q}/\mathbb{Z}$; and for a
complex place $v$, it is zero.
\end{theorem}

We can apply the above theorem to $k = \mathbb{Q}$. Picking the values
\[ \operatorname{inv}_2(\alpha) = \frac{1}{3},~~  \operatorname{inv}_3(\alpha) = -\frac{1}{3}, ~~ \operatorname{inv}_v(\alpha) = 0 \text{ for } v \neq 2, 3 \]
we get the existence of a non-zero element $\alpha \in \Br(\mathbb{Q})$ of order $3$, which is represented by a central division algebra $A$ over $\mathbb{Q}$.

Let us make two more comments that we will use.

\begin{itemize}
\item Let $k$ be an $E_\infty$-ring spectrum. Recall that a $k$-linear stable, idempotent complete $\infty$-category $\mathcal{C}$ is called \emph{smooth and proper} if it is dualizable as an object of $\Cat^\perf_{k}$. In this case, there is a natural equivalence
\[ \mathrm{map}_{\mathrm{Mot}_k}( \mathcal{U}_k( \mathcal{C}), \mathcal{U}_k( \mathcal{D}) ) \simeq K( \Fun^{\mathrm{ex}}_k( \mathcal{C}, \mathcal{D})),\]
see \cite[Theorem 5.25]{HoyoisScherotzkeSibilla2017HigherTraces}, where $\mathcal{U}_k : \Cat^\perf_{k} \rightarrow \mathrm{Mot}_k$ is the universal localizing invariant for $k$-linear stable idempotent complete $\infty$-categories.
\item If $A$ is a finite-dimensional central division algebra over $\mathbb{Q}$, it holds that $A$ is smooth and proper, by the observation that $\mathrm{Perf}_{A^{\op} \otimes_\mathbb{Q} A} \simeq \mathrm{Perf}_\mathbb{Q}$. Write $[M,N] = \pi_0 \mathrm{map}_{\mathrm{Mot}_k}( M, N)$ for $M, N$ in $\mathrm{Mot}_\mathbb{Q}$. Applying the previous result, we see that $ [ \mathcal{U}_\mathbb{Q}(A), - ] \simeq  K_0( \mathrm{Fun}^{\mathrm{ex}}( \mathrm{Perf}_A, - )).$
\end{itemize}

\begin{proposition} \label{notofordertwo}
Let $A$ be a finite-dimensional central division algebra over $\mathbb{Q}$ that is not of order $2$ in $\mathrm{Br}(\mathbb{Q})$. Then $\mathcal{U}_\mathbb{Q}(A) \not\simeq \mathcal{U}_\mathbb{Q}(A^{\op})$ in $\mathrm{Mot}_\mathbb{Q}$.
\end{proposition}

\begin{proof}
The condition that $A$ is not of order $2$ in $\mathrm{Br}(\mathbb{Q})$ is  equivalent to saying that $ A \otimes_\mathbb{Q} A \cong M_n(D)$ for some central division algebra $D \neq \mathbb{Q}$. In particular, $d = \mathrm{dim}_\mathbb{Q}(D) > 1$.

The composition
\[ [ \mathcal{U}_\mathbb{Q}(A^{\op}),  \mathcal{U}_\mathbb{Q}(A) ] \times [ \mathcal{U}_\mathbb{Q}(A),  \mathcal{U}_\mathbb{Q}(A^{\op}) ] \rightarrow [ \mathcal{U}_\mathbb{Q}(A), \mathcal{U}_\mathbb{Q}(A) ] \]
is identified with
$$ K_0( A \otimes_\mathbb{Q} A  ) \times K_0( A^{\op} \otimes_\mathbb{Q} A^{\op}  ) \rightarrow K_0( A^{\op} \otimes_\mathbb{Q} A  )  \cong K_0( \mathbb{Q}  ) \cong \mathbb{Z}.$$
The group  $K_0( A \otimes_\mathbb{Q} A  )$ is generated by $[D]$ and the group $K_0( A^{\op} \otimes_\mathbb{Q} A^{\op}  )$ is generated by $[D^{\op}]$. The product of these two elements is $[M_{d}(\mathbb{Q})] = d^2[1]$ in $K_0( \mathbb{Q}  )$. Hence the multiplication map identifies with
$$ \begin{array}{rcl}
\mathbb{Z} \times \mathbb{Z} & \rightarrow & \mathbb{Z} \\
(n, m) & \mapsto & d^2 nm,
\end{array} $$
which does not hit $1 \in \mathbb{Z}$, since $d > 1$. In other words, the identity in $[ \mathcal{U}_\mathbb{Q}(A), \mathcal{U}_\mathbb{Q}(A) ]$ does not lie in the image of the composition, and $\mathcal{U}_\mathbb{Q}(A)$ cannot be equivalent to $\mathcal{U}_\mathbb{Q}(A^{\op})$ in $\mathrm{Mot}_\mathbb{Q}$.
\end{proof}

The unit map $\mathbb{S} \rightarrow \mathbb{Q}$ exhibits $\mathbb{Q}$ as an idempotent algebra in $\Sp$. Therefore it holds that the tensor product on $\Sp$ restricts to the usual tensor product on $\mathrm{LMod}_\mathbb{Q}$ and $A \otimes_\mathbb{S} A' \simeq A \otimes_\mathbb{Q} A'$ for any two $\mathbb{Q}$-algebras $A, A'$, \cite[Proposition 4.8.2.7]{lurieha}. In particular, it holds that $A \otimes_\mathbb{S} \mathbb{Q} \simeq A$.

\begin{proof}[Proof of Theorem \ref{noninvariance}]
The claim $\mathcal{U}(A) \not\simeq \mathcal{U}(A^{\op})$ follows at once from the previous proposition, as otherwise we would obtain $\mathcal{U}_\mathbb{Q}(A) \simeq \mathcal{U}_\mathbb{Q}(A^{\op})$ via base change along $\mathbb{S} \rightarrow \mathbb{Q}$.
\end{proof}

\begin{remark}
The proof strategy given here can be generalized substantially. In fact, for every field $k$, the map
$$ \mathcal{U}_k : \Br(k) \rightarrow \mathrm{Pic}(\mathrm{Mot}_k)$$
is injective, see \cite[Theorem 9.1]{Tabuada2014Additive}, also \cite{Ramzi2025Localizing} for a more structural overview. The map $\mathcal{U}_k$ sends inverses to $(-)^{\op}$, thus giving Proposition \ref{notofordertwo} as a corollary. One can further generalize by replacing $k$ by a qcqs scheme $X$ to produce an even richer supply of counterexamples, see \cite[p.\ 3]{Ramzi2025Localizing}.
\end{remark}

\begingroup
\setlength{\emergencystretch}{8em}
\printbibliography
\endgroup

\end{document}